\documentclass[a4paper]{amsart}
\usepackage[all]{xy}
\usepackage{amsfonts}
\usepackage{amssymb}
\usepackage[centertags]{amsmath}
\usepackage{graphicx}
\usepackage{subfigure}
\usepackage{amssymb,amsthm,enumerate,epsfig}

\newcommand{\dis}{\displaystyle}

\def\N{{\mathbb N}}

\def\A{{\mathbb A}}

\def\G{{\mathcal G}}

\def\Z{{\mathbb Z}}
\def\PP{{\mathbb P}}
\def\m{{\mathbf m}}

\def\x{{\mathbf x}}
\def\b{{\mathbf b}}

\def\LM{\textrm{\small LM}}
\def\NLM{\textrm{\small NLM}}

\theoremstyle{plain}
 \newtheorem{thm}{Theorem}[section]
 
 \newtheorem{cor}[thm]{Corollary}
 
 \newtheorem{prop}[thm]{Proposition}
 \theoremstyle{definition}
 \newtheorem{defn}[thm]{Definition}
 \theoremstyle{remark}
 \newtheorem{rem}[thm]{Remark}
 \newtheorem{eg}[thm]{Example}
 \numberwithin{equation}{section}

\begin{document}

\title[Extensions]{Extensions of Toric Varieties}


\author[Mesut \c{S}ah\.{i}n]{Mesut \c{S}ah\.{i}n}
\address{Department of Mathematics, \c{C}ank\i r\i ~Karatek\.{i}n University, 18100, \c{C}ank\i r\i, ~ Turkey}
\curraddr{} \email{mesutsahin@karatekin.edu.tr}
\thanks{}



\date{\today}


\begin{abstract}
In this paper, we introduce the notion of ``extension" of a toric variety and study its fundamental properties.
This gives rise to infinitely many toric varieties with a special property, such as being set theoretic complete intersection or
arithmetically Cohen-Macaulay (Gorenstein) and having a Cohen-Macaulay tangent cone or a local ring with non-decreasing Hilbert function,
from just one single example with the same property, verifying Rossi's conjecture for larger classes and extending some results appeared in literature.
\end{abstract}

\maketitle

\section{Introduction}

Toric varieties are rational algebraic varieties with special
combinatorial structures making them objects on the crossroads of
different areas such as algebraic statistics, dynamical systems,
hypergeometric differential equations, integer programming,
commutative algebra and algebraic geometry.

Affine extensions of a toric curve has been introduced for the
first time by Arslan and Mete \cite{pf} inspired by Morales' work
\cite{mor} and used to study Rossi's conjecture saying that
Gorenstein local rings has non-decreasing Hilbert functions.
Later, we have studied set-theoretic complete intersection problem
for projective extensions motivated by the fact that every
projective toric curve is an extension of another lying in one
less dimensional projective space \cite{sahin}. Our purpose here
is to emphasize the nice behavior of toric varieties (of
\textit{any} dimension this time) under the operation of
extensions and we hope that this approach will provide a rich
source of classes for studying many other conjectures and open
problems.

In the first part of the present paper we note that affine
extensions can be obtained by gluing semigroups and thus their
minimal generating sets can be obtained by adding a binomial, see
Proposition \ref{affine}. In the projective case a similar result
holds under a mild condition, see Proposition \ref{bad}, which is
not true in general by Example \ref{145} since projective
extensions are not always obtained by gluing. In particular, if we
start with a set theoretic complete intersection, arithmetically
Cohen-Macaulay or Gorenstein toric variety, then we obtain
\textit{infinitely} many toric varieties having the same property,
generalizing \cite{thoma}.

We devote the second part for the local study of extensions of
toric varieties. Namely, if a toric variety has a Cohen-Macaulay
tangent cone or at least its local ring has a non-decreasing
Hilbert function, then we prove that its nice extensions share
these properties supporting Rossi's conjecture for higher
dimensional Gorenstein local rings and extending results appeared
in \cite[Proposition 4.1]{arslan} and \cite[Theorem 3.6]{pf}.
Similarly, we show that if its local ring is of homogeneous type,
then so are the local rings of its extensions. Local properties of
toric varieties of higher dimensions have not been studied
extensively, although there is a vast literature about toric
curves, see \cite{patil-tamone,shibuta}, \cite{ams,tamone} and
references therein. This paper might be considered as a first
modest step towards the higher dimensional case.

\section{Prelimineries}

Throughout the paper, $K$ is an algebraically closed field of any
characteristic. Let $S$ be a subsemigroup of $\N^d$ generated by
$\m_1,\dots,\m_n$. If we set $\deg_S (x_i)=\m_i$, then $S-$degree
of a monomial is defined by
$$\deg_S (\mathbf{x}^\mathbf{b})=\deg_S (x_1^{b_1}\cdots x_n^{b_n})=b_1\m_1+\cdots+b_n\m_n \in S.$$
The toric ideal of $S$, denoted $I_S$, is the prime ideal in
$K[x_1,\dots,x_n]$ generated by the binomials
$\x^{\mathbf{a}}-\x^{\b}$ with
$\deg_S(\x^{\mathbf{a}})=\deg_S(\x^{\b})$. The set of zeroes in
$\A^n$ is called the toric variety of $S$ and is denoted by $V_S$.
The projective closure of a variety $V$ will be denoted by
$\overline{V}$ as usual and we write $\overline{S}$ for the
semigroup defining the toric variety $\overline{V}_S$.

Denote by $S_{\ell,\m}$ the affine semigroup generated by $\ell
\m_1,\dots,\ell \m_n$ and $\m$, where $\ell$ is a positive
integer. When $\m \in S$, we define $\delta (\m)$ (respectively
$\Delta (\m)$) to be the minimum (respectively maximum) of all the
sums $s_1+\dots + s_n$ where $s_1,\dots,s_n$ are some non-negative
integers such that $\m=s_1\m_1+\dots+s_n\m_n$.

\begin{defn}[Extensions] With the preceding notation, we say that the affine toric variety $V_{S_{\ell,\m}}\subset \A^{n+1}$ is an \textit{extension} of
$V_S \subset \A^n$, if $\m \in S$, and $\ell$ is a positive
integer relatively prime to a component of $\m$. A projective
variety $\overline{E}\subset \PP^{n+1}$ will be called an
\textit{extension} of another one $\overline{X} \subset \PP^{n}$
if its affine part $E$ is an extension of the affine part $X$ of
$\overline{X}$.
\end{defn}

\begin{rem}\begin{enumerate}
             \item Notice that $\overline{V_S}=V_{\overline{S}}$, $I_S \subset I_{S_{\ell,\m}}$ and $I_{\overline{S}} \subset I_{\overline{S}_{\ell,\m}}$.
             \item The question of whether or not $I_{S_{\ell,\m}}$ (resp. $I_{\overline{S}_{\ell,\m}}$) has a minimal generating set containing a minimal generating set of $I_S$ (resp. $I_{\overline{S}}$) is not trivial.
             \item This definition generalizes the one given for monomial curves in \cite{pf,sahin}.
             \item In \cite{thoma}, special extensions for which $\ell$ equals to a multiple of $\delta(\m)$ has been studied without referring to them as extensions.
           \end{enumerate}
\end{rem}

Now we recall the definition of gluing semigroups introduced first
by Rosales \cite{rosales} and used by different authors to produce
certain family of examples in different context, see for example
\cite{ams,ojeda,mt}. Let $T=T_1\bigsqcup T_2$ be a decomposition
of a set $T \subset \N^d$ into two disjoint proper subsets. The
semigroup $\mathbb{N}T$ is called a \textit{gluing} of
$\mathbb{N}T_1$ and $\mathbb{N}T_2$ if there exists a nonzero
$\alpha \in \mathbb{N}T_1 \bigcap \mathbb{N}T_2$ such that
$\mathbb{Z}\alpha=\mathbb{Z}T_1 \bigcap \mathbb{Z}T_2$.

\begin{rem} \label{betti} If $S$ is a gluing of $S_1$ and
$S_2$ then $I_S=I_{S_1}+I_{S_2}+\langle F_{\alpha} \rangle$, where
$F_{\alpha}=x_1^{b_1}\cdots x_n^{b_n}-y_1^{c_1}\cdots y_n^{c_n}$
with $\deg_S(F_{\alpha})=\deg_S (x_1^{b_1}\cdots x_n^{b_n})=\deg_S
(y_1^{c_1}\cdots y_n^{c_n})=\alpha$. Since $F_{\alpha}$ is a
non-zero divisor, the minimal free resolution of $I_S$ can be
obtained by tensoring out the given minimal free resolutions of
$I_{S_1}$ and $I_{S_2}$, and then applying the mapping cone
construction. It is also standard to deduce that the coordinate
ring of $V_S$ is Cohen-Macaulay (Gorenstein) when the coordinate
rings of $V_{S_1}$ and $V_{S_2}$ are so. The converse is false as
there are Cohen-Macaulay (Gorenstein) toric curves in $\A^4$ which
can not be obtained by gluing two toric curves.
\end{rem}

The first observation is that affine extensions can be obtained by
gluing.

\begin{prop}\label{affine} If the toric variety $V_{S_{\ell,\m}}\subset \A^{n+1}$ is an extension of $V_S \subset \A^n$, then $S_{\ell,\m}$ is the gluing of
\,$\N T_1$ and $\N T_2$, where $T_1=\{\ell \m_1,\dots,\ell \m_n\}$
and $T_2=\{\m\}$. Consequently, $I_{S_{\ell,\m}}=I_S+\langle F
\rangle$, where $F=x_{n+1}^{\ell}-x_1^{s_1}\cdots x_n^{s_n}$.
\end{prop}

\begin{proof} First of all, $S=\N\{\m_1,\dots,\m_n\}$, $S_{\ell,\m}=\N T$, where the set $T=T_1\sqcup T_2$, $T_1=\{\ell \m_1,\dots,\ell \m_n\}$ and $T_2=\{\m\}$.
We claim that $S_{\ell,\m}$ is the gluing of its subsemigroups $\N
T_1$ and $\N T_2$. To this end we show that $\Z T_1 \cap \Z T_2=\Z
\alpha$, where $\alpha=\ell \m \in \N T_1 \cap \N T_2$.

Since $\ell \m=s_1 \ell \m_1+\cdots+s_n \ell \m_n$ with
non-negative integers $s_i$, we have clearly $\Z T_1 \cap \Z T_2
\supseteq \Z \alpha$. Take $z\m=z_1 \ell\m_1+\cdots+z_n\ell
\m_n\in \Z T_1 \cap \Z T_2$ and note that $z\m=\ell(z_1
\m_1+\cdots+z_n \m_n)$. Since $\ell$ is relatively prime to a
component of $\m$ by assumption, it follows that $\ell$ divides
$z$ and thus $z\m\in \Z \alpha$ yielding $\Z T_1 \cap \Z T_2
\subseteq \Z \alpha$. By the relation between the corresponding
ideals, we have $I_{S_{\ell,\m}}=I_S+\langle F \rangle$, since
$I_{T_1}=I_S$ and $I_{T_2}=0$.
\end{proof}

Since $F=x_{n+1}^{\ell}-x_1^{s_1}\cdots x_n^{s_n}$ is a non-zero
divisor of $R[x_{n+1}]/I_S R[x_{n+1}]$, where $R=K[x_1, \dots,
x_n]$, and $\dis
\sqrt{I_{S_{\ell,\m}}}=\sqrt{\sqrt{I_S}+\sqrt{F}}$ the following
is immediate.

\begin{cor} If $V_S\subset \A^{n}$ is a set theoretic complete intersection, arithmetically Cohen-Macaulay (Gorenstein),
so are its extensions $V_{S_{\ell,\m}}\subset \A^{n+1}$.
\end{cor}

\subsection{Projective Extensions} \mbox{} \par

Since projective extensions can not be obtained by gluing in
general, see \cite{sahin}, we study them separately in this
section. Contrary to the case of affine extensions, it is not true
in general that a minimal generating set of a projective extension
of $V_{\overline{S}}$ contains a minimal generating set of
$I_{\overline{S}}$ as illustrated by the following example.

\begin{eg} \label{145} If $S=\N\{1,4,5\}$, then the projective monomial curve $V_{\overline{S}}$ in $\PP^3$ is defined by $\overline{S}=\N\{(5,0),(4,1),(1,4),(0,5)\}$. Consider the projective extension $V_{\overline{S}_{1,10}}$ defined by the semigroup $$\overline{S}_{1,10}=\N\{(10,0),(9,1),(6,4),(5,5),(0,10)\}.$$
It is easy to see (use e.g. Macaulay \cite{mac}) that the set
$\{F_1,F_2,F_3,F_4,F_5\}$ constitutes a reduced Gr\"{o}bner basis
(and a minimal generating set) for the ideal $I_{\overline{S}}$
with respect to the reverse lexicographic order with
$x_1>x_2>x_3>x_0$, where
$$\begin{array}{lllll}
F_1 &=& x_1^4-x_0^3x_2 \\
F_2 &=& x_2^4-x_1x_3^3 \\
F_3 &=& x_1^2x_3^2-x_0x_2^3 \\
F_4 &=& x_1^3x_3-x_0^2x_2^2 \\
F_5 &=& x_1x_2-x_0x_3.
\end{array}$$

A computation shows that the set $\{F_1,F_4,F_5,F,F_6,F_7\}$ is a
reduced Gr\"obner basis for $I_{\overline{S}_{1,10}}$ with respect
to the reverse lexicographic order with $x_1>x_2>x_3>x_4>x_0$,
where
$$\begin{array}{lllll}
F&=&x_3^2-x_0x_4\\
F_6 &=& x_2^3-x_1^2x_4 \\
F_7 &=& x_1^3x_4-x_0x_2^2x_3.
\end{array}$$

We observe now that $F_7=x_2^2F_5-x_1F_6$ and that the set
$\{F_1,F_4,F_5,F,F_6\}$ is a minimal generating set of
$I_{\overline{S}_{1,10}}$. The fact that no minimal generating set
of $I_{\overline{S}}$ extends to a minimal generating set of
$I_{\overline{S}_{1,10}}$ follows from the observation that
$\mu(I_{\overline{S}})=\mu(I_{\overline{S}_{1,10}})(=5)$, where
$\mu(\cdot)$ denotes the minimal number of generators.

\end{eg}

Notice that the previous example reveals why minimal generating
sets need not extend when $\ell < \delta(\m)$. Next, we show that
this can be avoided as long as $\ell \geq \delta(\m)$.  So, we
compute a Gr\"obner basis for $I_{\overline{S}_{\ell,\m}}$ using
the Proposition \ref{affine} and the fact that if $\G$ is a
Gr\"obner basis for the ideal of an affine variety with respect to
a term order refining the order by degree, then the homogenization
of $\G$ is a Gr\"obner basis for the ideal of its projective
closure.

\begin{prop}\label{bad} If $\G$ is a reduced Gr\"obner basis for $I_{\overline{S}}$ with respect to a term order $\succ$ making $x_0$ the smallest variable and $\ell \geq \delta(\m)$, then $\G \cup \{F\}$ is a reduced Gr\"obner basis for $I_{\overline{S}_{\ell,\m}}$ with respect to a term order refining $\succ$ and making $x_{n+1}$ the biggest variable and thus $\dis I_{\overline{S}_{\ell,\m}}=I_{\overline{S}}\,+\langle F \rangle$, where $F =x_{n+1}^{\ell}-x_0^{\ell-\delta(\m)}x_1^{s_1}\cdots x_n^{s_n}$.
\end{prop}

\begin{proof} Let $\G=\{F_1,\dots,F_k\}$. If we dehomogenize the polynomials in $\G$ by substituting $x_0=1$, we get a reduced Gr\"obner basis $\{G_1,\dots,G_k\}$ for $I_{S}$ with respect to $\succ$ which refines the order by degree. From Proposition \ref{affine}, we know that $I_{S_{\ell,\m}}=I_S+\langle G \rangle=\langle G_1,\dots,G_k, G \rangle$, where $G=F(1,x_1,\dots,x_n)$. Since $\LM(G_i) \in K[x_1,\dots,x_n]$ and $\LM(G)=x_{n+1}^{\ell}$, it follows that $gcd(\LM(G_i),\LM(G))=1$, for all $i$. This implies that the set $\{G_1,\dots,G_k, G\}$ is a Gr\"obner basis for $I_{S_{\ell,\m}}$ with respect to a term order refining the order by degree and $\succ$. Hence, their homogenizations constitute the required Gr\"obner basis for $I_{\overline{S}_{\ell,\m}}$ as claimed.

Now, if $\LM(F_i)$ does not divide
$\NLM(F):=x_0^{\ell-\delta(\m)}x_1^{s_1}\cdots x_n^{s_n}$, it
follows that $\G \cup \{F\}$ is reduced as $\G$ is also.
Otherwise, i.e., $\NLM(F)=\LM(F_i)x_0^{\ell-\delta(\m)}M$, for
some monomial $M$ in $K[x_1,\dots,x_n]$, we replace $\NLM(F)$ by
$T_ix_0^{\ell-\delta(\m)}M$, since $\deg_S (\LM (F_i))=\deg_S
(T_i)$, which means that the new binomial
$F=x_{n+1}^{\ell}-T_ix_0^{\ell-\delta(\m)}M \in
I_{\overline{S}_{\ell,\m}}$. Since $\G$ is reduced and $F_i$ are
irreducible binomials, no $LM(F_j)$ divides
$T_ix_0^{\ell-\delta(\m)}M$. Therefore, the set $\G \cup \{F\}$ is
reduced as desired. Thus, we obtain
$I_{\overline{S}_{\ell,\m}}=I_{\overline{S}}+\langle F \rangle$.
\end{proof}

As in the affine case we have the following.
\begin{cor} If $V_{\overline{S}}\subset \PP^{n}$ is a set theoretic complete intersection, arithmetically Cohen-Macaulay (Gorenstein), so are its extensions $V_{\overline{S}_{\ell,\m}}\subset \PP^{n+1}$ provided that $\ell \geq \delta(\m)$.
\end{cor}

\section{Local Properties of Extensions}
In this section, we study Cohen-Macaulayness of tangent cones of
extensions of a toric variety having a Cohen-Macaulay tangent
cone, see \cite{arslan,patil-tamone,shibuta} for the literature
about Cohen-Macaulayness of tangent cones. We also show that if
the local ring of a toric variety is of homogeneous type or has a
non-decreasing Hilbert function, then its extensions share the
same property. As a main result, we demonstrate that in the
framework of extensions it is very easy to create infinitely many
new families of arbitrary dimensional and embedding codimensional
local rings having non-decreasing Hilbert functions supporting
Rossi's conjecture. This is important, as the conjecture is known
only for local rings with small (co)dimension:

\begin{itemize}
  \item Cohen-Macaulay rings of dimension $1$ and embedding codimension $2$, \cite{elias},
  \item Some Gorenstein rings of dimension $1$ and embedding codimension $3$, \cite{pf},
  \item Complete intersection rings of embedding codimension $2$, \cite{puthenpurakal},
  \item Some local rings of dimension $1$, \cite{ams,tamone},
\end{itemize}
where embedding codimension of a local ring is defined to be the
difference between its embedding dimension and dimension. For
instance, if $\A^n$ is the smallest affine space containing $V_S$,
then embedding dimension of the local ring of $V_S$ is $n$. Its
dimension coincides with the dimension of $V_S$ and its embedding
codimension is nothing but the codimension of $V_S$, i.e. $n-\dim
V_S$.

Before going further, we need to recall some terminology and
fundamental results which will be used subsequently. If $V_S
\subset \A^n$ is a toric variety, its associated graded ring is
isomorphic to $K[x_1,\dots,x_n]/{I_S}^*$, where ${I_S}^*$ is the
ideal of the tangent cone of $V_S$ at the origin, that is the
ideal generated by the polynomials $f^*$ with $f\in I_S$ and $f^*$
being the homogeneous summand of $f$ of the smallest degree. Thus,
the tangent cone is Cohen-Macaulay if this quotient ring is also.
Similarly, we can study the Hilbert function of the local ring
associated to $V_S$ by means of this quotient ring, since the
Hilbert function of the local ring is by definition the Hilbert
function of the associated graded ring. Finally, we can find a
minimal generating set for ${I_S}^*$ by computing a minimal
standard basis of $I_S$ with respect to a local order. For further
inquiries and notations to be used, we refer to \cite{singular}.

Assume now that $V_{S_{\ell,\m}} \subset A^{n+1}$ is an extension
of $V_S$, for suitable $\ell$ and $\m$. Then, by Proposition
\ref{affine}, we know that $I_{S_{\ell,\m}}=I_S+\langle F
\rangle$, where $F=x_{n+1}^{\ell}-x_1^{s_1}\cdots x_n^{s_n}$.
\begin{prop}\label{stdbasis} If $\G$ is a minimal standard basis of $I_S$ with respect to a negative degree reverse lexicographic ordering $\succ$ and $\ell \leq \Delta(\m)$, then $\G \cup \{F\}$ is a minimal standard basis of $I_{S_{\ell,\m}}$ with respect to a negative degree reverse lexicographic ordering refining $\succ$ and making $x_{n+1}$ the biggest variable.
\end{prop}
\begin{proof} Let $\G'=\G \cup \{F\}$. Since $NF({\rm spoly}(f,g) \vert \G)=0$, for all $f,g\in \G$, we have $NF({\rm spoly}(f,g) \vert \G')=0$.
Since $\LM(f) \in K[x_1,\dots,x_n]$ and $\LM(F)=x_{n+1}^{\ell}$,
it follows at once that $gcd(\LM(f),\LM(F))=1$, for every $f\in
\G$. Thus, we get $NF({\rm spoly}(f,F) \vert \G')=0$, for any
$f\in \G$. This reveals that $\G'$ is a standard basis with
respect to the afore mentioned local ordering and it is minimal
because of the minimality of $\G$.
\end{proof}

\begin{thm}\label{cone} If $V_S \subset A^{n}$ has a Cohen-Macaulay (Gorenstein) tangent cone at $0$, then so have its extensions $V_{S_{\ell,\m}} \subset A^{n+1}$, provided that $\ell \leq \Delta(\m)$.
\end{thm}
\begin{proof} An immediate consequence of the previous result is that  ${I_{S_{\ell,\m}}}^*={I_S}^*+\langle F^* \rangle$, where $F^*$ is $x_{n+1}^{\ell}$
whenever $\ell < \Delta(\m)$ and is $F$ if $\ell = \Delta(\m)$. In
any case $F^*$ is a nonzerodivisor on
$K[x_1,\dots,x_{n+1}]/{I_S}^*$ and thus
$K[x_1,\dots,x_{n+1}]/{I_{S_{\ell,\m}}}^*$ is Cohen-Macaulay as
required. In particular, both tangent cones have the same
Cohen-Macaulay type.
\end{proof}

\begin{rem} Theorem \ref{cone} generalizes the results appeared in \cite[Proposition 4.1]{arslan} and \cite[Theorem 3.6]{pf} from toric curves to toric varieties of any dimension. Moreover, Hilbert functions of the local rings of these extensions are nondecreasing in this case supporting Rossi's conjecture.
\end{rem}

According to \cite{HRV}, a local ring is of homogeneous type if
its Betti numbers coincide with the Betti numbers of its
associated graded ring, considered as a module over itself. It is
interesting to obtain local rings of homogeneous type, since in
this case, for example, the local ring and its associated ring
will have the same depth and their Cohen-Macaulayness will be
equivalent since they always have the same dimension. It will also
be easier to get information about the depth of the symmetric
algebra in this case, see \cite{HRV,rossi}.

\begin{prop}\label{hom} If the local ring of $V_S \subset A^{n}$ is of homogeneous type, then its extensions will also have local rings of homogeneous type if and only if $\ell \leq \Delta(\m)$.
\end{prop}
\begin{proof} Let $K[[S]]$ denote the local ring of $V_{S}$, i.e. the localization of the semigroup ring $K[S]=R/I_S$ at the origin, where $R=K[x_1,\dots,x_n]$. The Betti numbers of $K[[S]]$ and $K[S]$ is the same, since localization is flat. For the convenience of notation let us use $GR[S]$ for the associated graded ring corresponding to $V_S$ and $\beta_i(GR[S])$ for the Betti numbers of the minimal free resolution of $GR[S]=R/{I_S}^*$ over $R$.

Assume now that $K[[S]]$ is of homogeneous type, i.e.
$\beta_i(K[[S]])=\beta_i(GR[S])$, for all $i$. For any extension
$V_{S_{\ell,\m}} \subset A^{n+1}$ of $V_S$, we have from
Proposition \ref{affine} that $I_{S_{\ell,\m}}=I_S+\langle F
\rangle$, where $F=x_{n+1}^{\ell}-x_1^{s_1}\cdots x_n^{s_n}$.
Therefore, by Remark \ref{betti}, the Betti numbers are as follows

\begin{itemize}
   \item $\beta_1(K[[S_{\ell,\m}]])=\beta_1(K[[S]])+1$
   \item $\beta_i(K[[S_{\ell,\m}]])=\beta_i(K[[S]])+\beta_{i-1}(K[[S]]), \:2 \leq i \leq d=pd(K[[S]])$
   \item $\beta_{d+1}(K[[S_{\ell,\m}]])=\beta_d(K[[S]])$.
 \end{itemize}

If furthermore $\ell \leq \Delta(\m)$, Proposition \ref{stdbasis}
yields ${I_{S_{\ell,\m}}}^*={I_S}^*+\langle F^* \rangle$. Hence,
by Remark \ref{betti}, Betti numbers of $GR[S_{\ell,\m}]$ are
found as:

\begin{itemize}
   \item $\beta_1(GR[S_{\ell,\m}])=\beta_1(GR[S])+1$
   \item $\beta_i(GR[S_{\ell,\m}])=\beta_i(GR[S])+\beta_{i-1}(GR[S]), \:2 \leq i \leq d=pd(K[[S]])$
   \item $\beta_{d+1}(GR[S_{\ell,\m}])=\beta_d(GR[S])$.
 \end{itemize}
 It is obvious now that $\beta_i(GR[S_{\ell,\m}])=\beta_i(K[[S_{\ell,\m}]])$ for any $i$ and that local rings of extensions are of homogeneous type.

 The converse is rather trivial, since homogeneity of local rings of extensions force that $\beta_1(GR[S_{\ell,\m}])=\beta_1(K[[S_{\ell,\m}]])$, i.e. ${I_{S_{\ell,\m}}}^*={I_S}^*+\langle F^* \rangle$ which is possible only if $\ell \leq \Delta(\m)$.
\end{proof}

Finally, inspired by \cite[Theorem 3.1]{ams}, we consider
extensions of a toric variety whose local ring has a
non-decreasing Hilbert function and whose tangent cone is not
necessarily Cohen-Macaulay. The proof is a modification of that of
\cite[Theorem 3.1]{ams} and the reason for this is that there are
toric surfaces having non-decreasing Hilbert functions but having
Hilbert series expressed as a ratio of a polynomial with some
\textbf{negative} coefficients. The Hilbert series of the toric
variety in Example $3.6$ item ($3$) is such an example:
$$(1 + 3t + 6t^2 + 8t^3 + 9t^4 + 7t^5 + 3t^6 - t^8) / (1-t)^2.$$
\begin{thm}\label{hf} If $V_S \subset A^{n}$ has a local ring with non-decreasing Hilbert function, then so have its extensions $V_{S_{\ell,\m}} \subset A^{n+1}$, provided that $\ell \leq \Delta(\m)$.
\end{thm}
\begin{proof} Let $R=K[x_1,\dots,x_n]$. If $I$ is a graded ideal of $R$, then it is a standard fact that the Hilbert function of $R/I$ is just the Hilbert function of $R/\LM(I)$, where $\LM(I)$ is a monomial ideal consisting of the leading monomials of polynomials in $I$. Now, Proposition \ref{stdbasis} reveals that ${I_{S_{\ell,\m}}}^*={I_S}^*+\langle F^* \rangle$, where $F=x_{n+1}^{\ell}-x_1^{s_1}\cdots x_n^{s_n}$ and that $\LM({I_{S_{\ell,\m}}}^*)=\LM({I_S}^*)+\langle \LM(F^*) \rangle$. Since $\LM({I_S}^*)\subset R$ and $\LM(F^*)=x_{n+1}^{\ell}$ with respect to the local order mentioned in Proposition \ref{stdbasis}, it follows from the proof of \cite[Proposition 2.4]{bayer} that $R'=R_1 \otimes_K R_2$, where $$R'=R[x_{n+1}]/\LM({I_{S_{\ell,\m}}}^*),\,\,R_1=R/\LM({I_S}^*)\,\,\mbox{and}\,\, R_2=K[x_{n+1}]/\langle x_{n+1}^{\ell} \rangle.$$
Hilbert series of $R_1$ can be given as $\sum_{k \geq 0} a_k t^k$,
where $a_k \leq a_{k+1}$ for any $k\geq 0$, since from the
assumption the local ring associated to $V_S$ has non-decreasing
Hilbert function. It is clear that the Hilbert series of $R_2$ is
$h_2(t)=1+t+\dots+t^{\ell-1}$. Since the Hilbert series of $R'$ is
the product of those of $R_1$ and $R_2$, we observe that the
Hilbert series of $R'$ is given by
\begin{eqnarray*}\sum_{k \geq 0} b_k t^{k}&=&(1+t+\cdots+t^{\ell-1})\sum_{k \geq 0} a_k t^k\\
&=&\sum_{k \geq 0} a_k t^k+\sum_{k \geq 0} a_k t^{k+1}+\cdots+\sum_{k \geq 0} a_k t^{k+\ell-1}\\
&=&\sum_{k \geq 0} a_k t^k+ \sum_{k \geq 1} a_{k-1} t^k
+\cdots+\sum_{k \geq \ell-1} a_{k-\ell+1} t^k.
\end{eqnarray*}
Therefore, the Hilbert series $\sum_{k \geq 0} b_k t^{k}$ of $R'$
is given by
$$a_0+(a_0+a_1)t+\cdots+(a_0+\cdots+a_{\ell-2})t^{\ell-2}+ \sum_{k \geq \ell-1} (a_k+a_{k-1}+\cdots+a_{k-\ell+1}) t^k.$$
It is now clear that $b_k \leq b_{k+1}$, for any $0\leq k \leq
\ell-2$, from the first part of the last equality above, since
$a_k \leq a_{k+1}$. For all the other values of $k$, i.e. $k\geq
\ell-1$, we have $b_k-b_{k+1}=a_{k-\ell+1}-a_{k+1}\leq 0$ which
accomplishes the proof.
\end{proof}
\begin{eg} In the following, we will say that the extension is \textit{nice} if $\ell \leq \Delta(\m)$.
\begin{enumerate}
            \item The local ring of the affine cone of a projective toric variety is always of homogeneous type, for instance, $S=\{(3,0),(2,1),(1,2),(0,3)\}$ defines a projective toric curve in $\PP^3$ and its affine cone is the toric surface $V_S \subset \A^4$ with the homogeneous toric ideal $I_S=\langle x_2^2-x_1x_3, x_3^2-x_2x_4, x_2x_3-x_1x_4\rangle$. Thus by Proposition \ref{hom}, its affine nice extensions will have homogeneous type local rings which are not necessarily homogeneous. Take for example, $\ell=1$ and $\m=(0,3s)$ for any $s>1$. Then, although $I_{S_{\ell,\m}}=I_S+\langle x_4^s-x_5\rangle$ is not homogeneous, its local ring is of homogeneous type.
            \item Similarly, one can produce Cohen-Macaulay tangent cones using arithmetically Cohen-Macaulay projective toric varieties, since the toric ideal $I_S$ of their affine cones are homogeneous and thus $I_S={I_S}^*$. Therefore, all of their affine nice extensions will have Cohen-Macaulay tangent cones and local rings with non-decreasing Hilbert functions, by Theorem \ref{cone}. The toric variety $V_S \subset \A^4$ considered in the previous item (1) and its nice extensions illustrate this as well.
            \item Take $S=\{(6,0),(0,2),(7,0),(6,4),(15,0)\}$. Then it is easy to see that $I_S=\langle x_1x_2^2-x_4, x_3^3-x_1x_5, x_1^5-x_5^2\rangle$. Since $V_S \subset \A^5$ is a toric surface of codimension $3$, $I_S$ is a complete intersection and thus the local ring of $V_S$ is Gorenstein. But, the tangent cone at the origin, is determined by ${I_S}^*=\langle x_5^2,x_4, x_3^3x_5,x_3^6,x_1x_5\rangle$ and thus is not Cohen-Macaulay. Nevertheless, its Hilbert function $H_S$ is non-decreasing:\\
                $H_S(0)=1, H_S(1)=4, H_S(2)=8, H_S(3)=13, H_S(r)=6r-6, \: \mbox{for}\: r\geq 4.$
                Consider now all nice extensions of $V_S$; defined by the following semigroups $S_{\ell,\m}=\{(6\ell,0),(0,2\ell),(7\ell,0),(6\ell,4\ell),(15\ell,0),\m\}$.
                Therefore, Theorem \ref{hf} produces infinitely many new toric surfaces with local rings of dimension $2$ and embedding codimension $4$ whose Hilbert functions are non-decreasing even though their tangent cones are not Cohen-Macaulay. Indeed, one may produce this sort of examples in any embedding codimension by taking a sequence of nice extensions of the same example, since in each step the embedding codimension increases by one.
          \end{enumerate}
\end{eg}


\end{document}